\documentclass{scrartcl}
\usepackage[english]{babel}
\usepackage{enumitem}
\usepackage{hyperref}
\usepackage[T1]{fontenc}
\usepackage{lmodern}
\usepackage[utf8]{inputenc}
\usepackage{amsmath}
\usepackage{amssymb}
\usepackage{amsthm}
\usepackage[babel]{csquotes}
\usepackage{tikz}
\usepackage{bm}
\usepackage{upgreek}
\usepackage{subfigure}
\usepackage{verbatim}
\usepackage{textcomp}
\usepackage{wasysym}
\usepackage{combelow}
\usepackage{booktabs}
\usepackage{epstopdf}
\usepackage{pgfplots}
\pgfplotsset{compat=1.5}

\newcommand{\p}[1]{\ensuremath{\mathord{\left(#1\right)}}}
\newcommand{\abs}[1]{\ensuremath{\left\vert#1\right\vert}}
\newcommand{\set}[1]{\ensuremath{\left\lbrace #1 \right\rbrace}}
\newcommand{\setcond}[2]{\ensuremath{\left\lbrace #1 \,\middle\vert \, #2 \right\rbrace}}
\newcommand{\inpr}[2]{\ensuremath{\left\langle #1, #2 \right\rangle}}
\newcommand{\RR}{\ensuremath{\mathbb{R}}}

\newcommand{\clint}[2]{\ensuremath{\left[#1, #2\right]}}
\newcommand{\opint}[2]{\ensuremath{\left(#1, #2\right)}}
\newcommand{\coint}[2]{\ensuremath{\left[#1, #2\right)}}
\newcommand{\ocint}[2]{\ensuremath{\left(#1, #2\right]}}
\newcommand{\defeq}{\ensuremath{\mathrel{\mathop:}=}}

\newcommand{\RRc}{\ensuremath{\overline{\mathbb{R}}}}

\newcommand{\Hh}{\mathcal{H}}
\newcommand{\Gg}{\mathcal{G}}

\newcommand{\norm}[1]{\left\Vert #1 \right\Vert}

\newcommand{\dd}[1]{\mathop{\mathrm{d} #1}}

\newcommand{\Id}{\ensuremath{\mathrm{Id}}}

\newcommand{\norel}{\ensuremath{\mathrel{\phantom{=}}}}

\newcommand{\dist}{\ensuremath{\mathrm{dist}}}
\newcommand{\crit}{\ensuremath{\mathrm{crit}}}

\DeclareMathOperator{\dom}{dom}

\DeclareMathOperator{\cl}{cl}

\DeclareMathOperator*{\argmin}{arg\,min}

\DeclareMathOperator{\Prox}{Prox}

\theoremstyle{plain}
\ifcsdef{lemma}{\relax}{\newtheorem{lemma}{Lemma}}
\ifcsdef{theorem}{\relax}{\newtheorem{theorem}{Theorem}}
\ifcsdef{corollary}{\relax}{}
\ifcsdef{proposition}{\relax}{\newtheorem{proposition}{Proposition}}

\theoremstyle{definition}
\ifcsdef{definition}{\relax}{\newtheorem{definition}{Definition}}
\ifcsdef{example}{\relax}{}
\ifcsdef{remark}{\relax}{\newtheorem{remark}{Remark}}
\ifcsdef{problem}{\relax}{}
\ifcsdef{algorithm}{\relax}{}

\title{A general double-proximal gradient algorithm for d.c. programming}
\author{Sebastian Banert \thanks{University of Vienna, Faculty of Mathematics, Oskar-Morgenstern-Platz 1, A-1090 Vienna, Austria,
email: sebastian.banert@univie.ac.at. Research supported by FWF (Austrian Science Fund), project I 2419-N32.} \and Radu Ioan Bo\cb{t} \thanks{University of Vienna, Faculty of Mathematics, Oskar-Morgenstern-Platz 1, A-1090 Vienna, Austria,
email: radu.bot@univie.ac.at. Research supported by FWF (Austrian Science Fund), project I 2419-N32.} }

\begin{document}
\maketitle

\begin{abstract}
The possibilities of exploiting the special structure of d.c. programs, which consist of optimizing the difference of convex functions, are currently more or less limited to variants of the DCA proposed by Pham Dinh Tao and Le Thi Hoai An in 1997. These
assume that either the convex or the concave part, or both, are evaluated by one of their subgradients. 

In this paper we propose an algorithm which allows the evaluation of both the concave and the convex part by their proximal points. Additionally, we allow a smooth part, which is evaluated via its gradient. 
In the spirit of primal-dual splitting algorithms, the concave part might be the composition of a concave function with a linear operator, which are, however, evaluated separately.

For this algorithm we show that every cluster point is a solution of the optimization problem. Furthermore, we show the connection to the Toland dual problem and prove a descent property for the objective function values of a primal-dual formulation of the problem. 
Convergence of the iterates is shown if this objective function satisfies the Kurdyka--\L ojasiewicz property. In the last part, we apply the algorithm to an image processing model.\vspace{1ex}

\noindent \textbf{Key Words.} d.c. programming, Toland dual, proximal-gradient algorithm, Kurdyka--\L ojasiewicz property, convergence analysis \vspace{1ex}

\noindent \textbf{AMS subject classification.} 90C26, 90C30, 65K05
\end{abstract}

\section{Introduction}\label{sec:intro}
Optimization problems where the objective function can be written as a difference of two convex functions arise naturally in several applications, such as image processing \cite{LouZengOsherXin:2015}, machine learning \cite{ThiNguyen:2016}, 
optimal transport \cite{Carlier:2008} and sparse signal recovering \cite{GassoRakotomamonjyCanu:2009}. Generally, the class of d.c. functions is rather broad and contains for example every twice continuously differentiable function. 
For an overview over d.c. functions, see e.g. \cite{HorstThoai:1999}.

The classical approach to iteratively find local extrema of d.c. problems was described by Tao and An \cite{TaoAn:1997} in 1997 under the name DCA (d.c. algorithms). One of the most recent papers on this topic is \cite{ArtachoFlemingVuong:2016},
where an accelerated variant of the DCA method is proposed under the supplementary assumption that both the convex and the concave part are continuously differentiable. In 2003, Sun, Sampaio and Candido introduced a proximal point approach into the theory of d.c. algorithm \cite{SunSampaioCandido:2003}, where the convex part is evaluated by its proximal point operator, while its concave part is still 
evaluated by one of its subgradients. Later on, the approach in \cite{SunSampaioCandido:2003} has been extended in \cite{MaingeMoudafi:2008, AnNam:2015, DinhKimJiao:2015} by considering in the convex part a further convex smooth summand that is evaluated via its gradient.

In this paper, we go one step further by proposing an algorithm, where both convex and concave part are evaluated via proximal steps. In convex optimisation, using proximal steps instead of subgradient steps has several advantages:
\begin{itemize}
  \item The subdifferential at a point may be a non-singleton set, in particular it may be empty or may consist of several distinct elements. In an algorithm, one may get stuck or have to choose one, respectively.
  \item Even if the subdifferential is a singleton in each step, it might be highly discontinuous, so small deviations might lead to a very different behaviour of the iterations.
  \item Better convergence rates can be guaranteed for proximal algorithms than for subgradient algorithms (compare \cite{BeckTeboulle:2009} and \cite[Theorem 3.2.3]{Nesterov:2004}).
\end{itemize}

In addition, we consider a linear operator in the concave part of the objective function, which is evaluated in a forward manner in the spirit of primal-dual splitting methods.

In Section \ref{sec:problem}, we present the problem to be solved together with its Toland dual and attach to them a primal-dual formulation in form of a minimization problem, too. 
We derive first-order optimality conditions and relate the optimal solutions and the critical points of the primal-dual minimization problems to the optimal solutions and, respectively, the critical points of both primal and dual optimization problems.

In Section \ref{sec:algorithm}, we propose a double-proximal d.c. algorithm, which generates both a primal and a dual sequence of iterates and show several properties which make it comparable to DCA. 
More precisely, we prove a descent property for the objective function values of a primal-dual formulation and that every cluster point of the sequence of primal iterates is a critical point of the primal problem,
while every critical point of the sequence of dual iterates is a critical point of the dual problem.

In Section \ref{sec:KL}, we show global convergence of our algorithm and convergence rates for the iterates in some certain cases, provided that the objective function of the primal-dual reformulation satisfies the Kurdyka--\L ojasiewicz property; in other words, it is a \emph{K\L{} function}. 
The convergence analysis relies on methods and concepts of real algebraic geometry introduced by \L{}ojasiewicz \cite{Lojasiewicz:1963} and Kurdyka \cite{Kurdyka:1998} and 
later developed  in the nonsmooth setting by Attouch, Bolte and Svaiter \cite{AttouchBolteSvaiter:2013} and Bolte, Sabach and Teboulle \cite{BolteSabachTeboulle:2014}. 
One of the remarkable properties of the K\L{} functions is their ubiquity in applications (see \cite{BolteSabachTeboulle:2014}). The class of  K\L{} functions contains semi-algebraic, real sub-analytic, semiconvex, uniformly convex and 
convex functions satisfying a growth condition. 

We close our paper with some numerical examples addressing an image deblurring and denoising problem in the context of different DC regularizations.

\subsection{Notation and preliminaries}

For the theory of convex analysis in finite-dimensional spaces, see the book \cite{Rockafellar:1970b}.
We shall consider functions taking values in the \emph{extended real line} $\RRc \defeq \RR \cup \set{+\infty, -\infty}$. We agree on the order $-\infty < a < +\infty$ for any real number $a$ and the operations
\begin{align*}
  +\infty + a = a + \infty = +\infty - \infty = -\infty + \infty = +\infty + \infty &= +\infty, \\
  -\infty + a = a - \infty = -\infty - \infty &= -\infty, \\
   0 \cdot \p{-\infty} = 0, 0 \cdot \p{+\infty} &= +\infty
  \end{align*}
for arbitrary $a \in \RR$ (see \cite{Zalinescu:2002}). Let $\Hh$ be a real finite-dimensional Hilbert space. For a function $f: \Hh \to \RRc$, we denote by
\[
  \dom f \defeq \setcond{x\in \Hh}{f\p{x} < +\infty}
\]
its \emph{domain}. The function $f$ is called \emph{proper} if it does not take the value $-\infty$ and $\dom f \neq \emptyset$. It is called \emph{convex} if
\[
  f\p{\p{1 - \lambda} x + \lambda y} \leq \p{1 - \lambda} f\p{x} + \lambda f\p{y}
\]
for all $x, y \in \Hh$ and $0 \leq \lambda \leq 1$. The \emph{conjugate function} $f^*: \Hh \to \RRc$ of $f: \Hh \to \RRc$ is defined by
\[
  f^*\p{x^*} = \sup\setcond{\inpr{x^*}{x} - f\p{x}}{x\in \Hh}.
\]
If $f$ is proper, convex and lower semicontinuous, then $f^{**} \defeq \p{f^*}^* = f$ by the Fenchel--Moreau theorem.

The \emph{convex subdifferential} $\partial f\p{x}$ at $x\in \Hh$ of a function $f: \Hh \to \RRc$ is empty if $x \notin \dom f$ and
\[
  \partial f\p{x} = \setcond{x^* \in \Hh}{\forall y\in \Hh: f\p{y} \geq f\p{x} + \inpr{x^*}{y - x}}
\]
otherwise. Let $\gamma > 0$ and $f: \Hh \to \RRc$ be proper, convex and lower semicontinuous. The \emph{proximal point} $\Prox_{\gamma f}\p{x}$ of $\gamma f$ at $x\in \Hh$ is defined as
\[
  \Prox_{\gamma f}\p{x} = \argmin\setcond{\gamma f\p{y} + \frac{1}{2} \norm{y - x}^2}{y \in \Hh}.
\]
The set of minimizers in the definition above is a singleton \cite[Proposition 12.15]{BauschkeCombettes:2011}, and the proximal point is characterised by the variational inequality \cite[Proposition 12.26]{BauschkeCombettes:2011}
\[
  f\p{y} \geq f\p{\Prox_{\gamma f}\p{x}} + \frac{1}{2\gamma} \inpr{y - \Prox_{\gamma f}\p{x}}{x - \Prox_{\gamma f}\p{x}}
\]
for all $y\in \Hh$, which is equivalent to
\begin{equation}\label{eq:prox_subdiff_inclusion}
  \frac{1}{\gamma} \p{x - \Prox_{\gamma f}\p{x}} \in \partial f\p{\Prox_{\gamma f}\p{x}}.
\end{equation}

When dealing with nonconvex and nonsmooth functions, we have to consider subdifferentials more general than the convex one. The \emph{Fr\'echet subdifferential} $\partial_F f\p{x}$ at $x\in \Hh$ of a proper and lower semicontinuous function $f: \Hh \to \RRc$ is empty if $x \notin \dom f$ and
\[
  \partial_F f\p{x} = \setcond{x^* \in \Hh}{\liminf_{\substack{y \to x \\ y \neq x}} \frac{f\p{y} - f\p{x} - \inpr{x^*}{y - x}}{\norm{y - x}} \geq 0}
\]
otherwise. The \emph{limiting (Mordukhovich) subdifferential} $\partial_L f\p{x}$ at $x\in \Hh$ of a proper and lower semicontinuous function $f: \Hh \to \RRc$ is empty if $x\notin \dom f$ and
\begin{multline*}
  \partial_L f\p{x} = \biggl\{x^* \in \Hh \, \bigg\vert \,\exists \p{x_k}_{k\geq 0}, \p{x^*_k}_{k\geq 0}: x_k \in \Hh, x^*_k \in \partial_F f\p{x_k}, k \geq 0, \\
  x_k \to x, f\p{x_k} \to f\p{x}, x^*_k \to x^* \text{ as } k\to +\infty\biggr\}
\end{multline*}
otherwise.

\section{Problem statement}\label{sec:problem}

Let $\Gg$ and $\Hh$ be real finite-dimensional Hilbert spaces, let $g: \Hh \to \RRc$ and $h: \Gg \to \RRc$ be proper, convex and lower semicontinuous functions, let $\varphi: \Hh \to \RR$ be a convex, 
Fr\'echet differentiable function with $\frac{1}{\beta}$-Lipschitz continuous gradient, for some $\beta > 0$, and let $K: \Hh \to \Gg$ be a linear mapping (and $K^*: \Gg \to \Hh$ its adjoint). We consider the problem
\begin{equation}\label{eq:DC_Problem}
  \min \setcond{g\p{x} + \varphi\p{x} - h\p{Kx}}{x\in \Hh}
\end{equation}
together with its Toland dual problem \cite{Toland:1979, Toland:1978}
\begin{equation}\label{eq:DC_Toland_dual}
  \min \setcond{h^*\p{y} - \p{g + \varphi}^*\p{K^* y}}{y\in \Gg}.
\end{equation}
The following primal-dual formulation will turn out to be useful in the sequel:
\begin{equation}\label{eq:DC_pd}
  \min \setcond{\Phi\p{x, y}}{x \in \Hh, y\in \Gg} \qquad \text{with } \Phi\p{x, y} \defeq g\p{x} + \varphi\p{x} + h^*\p{y} - \inpr{y}{Kx},
\end{equation}
where $\Phi: \Hh \times \Gg \to \RRc$ is proper and lower semicontinuous.

Let us derive necessary optimality conditions for the problems \eqref{eq:DC_Problem}, \eqref{eq:DC_Toland_dual} and \eqref{eq:DC_pd}:
\begin{proposition}\begin{enumerate}
    \item The optimal values of \eqref{eq:DC_Problem}, \eqref{eq:DC_Toland_dual} and \eqref{eq:DC_pd} are equal.
    \item For all $x \in \Hh$ and $y \in \Gg$,
      \begin{align*}
        \Phi\p{x, y} &\geq g\p{x} + \varphi\p{x} - h\p{Kx}\qquad \text {and} \\
        \Phi\p{x, y} &\geq h^*\p{y} - \p{g + \varphi}^*\p{y} \p{K^* y}.
      \end{align*}
    \item Let $\bar x \in \Hh$ be a solution of \eqref{eq:DC_Problem}. Then $\partial\p{h\circ K}\p{\bar x} \subseteq \partial g\p{\bar x} + \nabla\varphi\p{\bar x}$.
    \item Let $\bar y \in \Gg$ be a solution of \eqref{eq:DC_Toland_dual}. Then $\partial \p{\p{g + \varphi}^* \circ K^*}\p{\bar y} \subseteq \partial h^*\p{\bar y}$.
    \item Let $\p{\bar x, \bar y}\in \Hh \times \Gg$ be a solution of \eqref{eq:DC_pd}. Then $\bar x$ is a solution of \eqref{eq:DC_Problem}, and $\bar y$ is a solution of \eqref{eq:DC_Toland_dual}. Furthermore, the inclusions
      \begin{align}
        K^* \bar y &\in \partial g\p{\bar x} + \nabla \varphi\p{\bar x}, \label{eq:stat_incl1} \\
        K\bar x &\in \partial h^*\p{\bar y} \label{eq:stat_incl2}
      \end{align}
      hold.
  \end{enumerate}
\end{proposition}
\begin{proof}
  \begin{enumerate}
    \item By the Fenchel--Moreau theorem, applied to $h$, we have
      \begin{align*}
        &\norel \inf \setcond{g\p{x} + \varphi\p{x} - h\p{Kx}}{x\in \Hh} \\
        &= \inf \setcond{g\p{x} + \varphi\p{x} - h^{**}\p{Kx}}{x\in \Hh} \\
        &= \inf \setcond{g\p{x} + \varphi\p{x} - \sup \setcond{\inpr{y}{Kx} - h^*\p{y}}{y\in \Gg}}{x \in \Hh} \\
        &= \inf \setcond{g\p{x} + \varphi\p{x} + h^*\p{y} - \inpr{y}{Kx}}{x \in \Hh, y \in \Gg} \\
        &= \inf \setcond{h^*\p{y} - \sup\setcond{\inpr{x}{K^* y} - \p{g + \varphi}\p{x}}{x\in \Hh}}{y \in \Gg} \\
        &= \inf \setcond{h^*\p{y} - \p{g + \varphi}^*\p{K^* y}}{y \in \Gg}.
      \end{align*}
    \item Let $x \in \Hh$ and $y \in \Gg$. Then,
      \begin{align*}
        g\p{x} + \varphi\p{x} - h\p{Kx}
        &= g\p{x} + \varphi\p{x} - h^{**}\p{Kx} \\
        &= g\p{x} + \varphi\p{x} - \sup\setcond{\inpr{Kx}{\tilde y} - h^*\p{\tilde y}}{\tilde y \in \Gg} \\
        &\leq g\p{x} + \varphi\p{x} - \inpr{Kx}{y} + h^*\p{y},
      \end{align*}
      and the other inequality is verified by an analogous calculation.
    \item Let $\bar x \in \Hh$ be a solution of \eqref{eq:DC_Problem}, i.e.,
      \begin{equation}\label{eq:necessary_condition_solution}
        \forall x\in \Hh: g\p{x} + \varphi\p{x} - h\p{Kx} \geq g\p{\bar x} + \varphi\p{\bar x} - h\p{K\bar x}.
      \end{equation}
      If $h\p{K\bar x} = +\infty$, then, by definition, $\partial\p{h \circ K}\p{\bar x} = \emptyset$, and the inclusion automatically holds. If $h\p{K\bar x} < +\infty$, then the optimal value of \eqref{eq:DC_Problem} must be $> -\infty$, which implies
      \begin{equation}\label{eq:necessary_condition_boundedness}
        h\p{Kx} < +\infty \qquad \text{for all } x \in \dom g.
      \end{equation}
      Now let $y \in \partial \p{h \circ K}\p{\bar x}$. Then
      \begin{equation}\label{eq:necessary_condition_subdifferential}
        \forall x\in \Hh: h\p{Kx} \geq h\p{K\bar x} + \inpr{y}{x - \bar x}.
      \end{equation}
      Adding \eqref{eq:necessary_condition_solution} and \eqref{eq:necessary_condition_subdifferential} yields
      \begin{equation}\label{eq:necessary_condition_sum}
        \forall x\in \Hh: g\p{x} + \varphi\p{x} \geq g\p{\bar x} + \varphi\p{\bar x} + \inpr{y}{x - \bar x}.
      \end{equation}
      If $g\p{x} = +\infty$, then \eqref{eq:necessary_condition_sum} is automatically satisfied, otherwise $x \in \dom g$ and, by \eqref{eq:necessary_condition_boundedness}, $h\p{Kx} < +\infty$, and both sides of both \eqref{eq:necessary_condition_solution} and \eqref{eq:necessary_condition_subdifferential} are finite. In either case, we have shown $y \in \partial \p{g + \varphi}\p{\bar x} = \partial g\p{\bar x} + \nabla \varphi\p{\bar x}$.
    \item The proof of this statement is analogous.
    \item Let $\p{\bar x, \bar y}$ be a solution of \eqref{eq:DC_pd}. (In particular, if such a solution exists, the common optimal value of \eqref{eq:DC_Problem}, \eqref{eq:DC_Toland_dual} and \eqref{eq:DC_pd} must be finite.) The function $x \mapsto \Phi\p{x, \bar y}$ is convex and takes a minimum at $\bar x$. Therefore
      \[
        0 \in \partial g\p{\bar x} + \nabla \varphi\p{\bar x} - K^* \bar y,
      \]
      which proves \eqref{eq:stat_incl1}. The same argument works for the function $y \mapsto \Phi\p{\bar x, y}$ and implies
      \[
        0 \in \partial h^*\p{\bar y} - K\bar x,
      \]
      which is \eqref{eq:stat_incl2}. For these inclusions, we obtain equality in the Young--Fenchel inequality, i.e.,
      \begin{align*}
        \p{g + \varphi}^*\p{K^* \bar y} + \p{g + \varphi}\p{\bar x} &= \inpr{\bar x}{K^* \bar y}, \\
        h^*\p{\bar y} + h\p{K\bar x} &= \inpr{\bar y}{K\bar x}.
      \end{align*}
      Therefore,
      \begin{align*}
        \p{g + \varphi}\p{\bar x} - h\p{K\bar x} &= h^*\p{\bar y} - \p{g + \varphi}^*\p{K^* \bar y} \\
        &= h^*\p{\bar y} - \sup\setcond{\inpr{x}{K^*\bar y} - g\p{x} - \varphi\p{x}}{x\in \Hh} \\
        &\leq h^*\p{\bar y} + g\p{\bar x} + \varphi\p{\bar x} - \inpr{\bar x}{K^* \bar y}.
      \end{align*}
      Since $\p{\bar x, \bar y}$ is a solution of \eqref{eq:DC_pd}, the last expression equals the common optimal value of \eqref{eq:DC_Problem}, \eqref{eq:DC_Toland_dual} and \eqref{eq:DC_pd}. \qedhere
  \end{enumerate}
\end{proof}

\begin{definition}
  We say that $\p{\bar x, \bar y} \in \Hh \times \Gg$ is a \emph{critical point} of the objective function $\Phi$ of \eqref{eq:DC_pd} if the inclusions \eqref{eq:stat_incl1} and \eqref{eq:stat_incl2} are satisfied. We denote
  by  $\crit \Phi$ the set of critical points of the function $\Phi$.
\end{definition}
\begin{remark}\label{rem:critical}
If $\p{\bar x, \bar y} \in \Hh \times \Gg$ is a critical point of $\Phi$, then
\begin{align}
  K^* \bar y &\in K^* \partial h\p{K\bar x} \cap \p{\partial g\p{\bar x} + \nabla \varphi\p{\bar x}}, \label{eq:primal_critical} \\
  K \bar x &\in K \partial \p{g + \varphi}^* \p{K^* \bar y} \cap \partial h^*\p{\bar y}. \label{eq:dual_critical}
\end{align}
By adopting the terminology of e.g. \cite[p. 297]{TaoAn:1997}, we denote by
$$\crit \p{g + \varphi - h \circ K}:= \{x \in \Hh: K^* \partial h\p{K x} \cap \p{\partial g\p{x} + \nabla \varphi\p{x}} \neq \emptyset\}$$
the set of critical points of the objective function $g + \varphi - h \circ K$ of \eqref{eq:DC_Problem} and by
$$\crit\p{h^* - \p{g + \varphi}^* \circ K^*}:=  \{y \in \Gg: K \partial \p{g + \varphi}^* \p{K^*y} \cap \partial h^*\p{y} \neq \emptyset\}$$
the set of critical points of the objective function $h^* - \p{g + \varphi}^* \circ K^*$ of \eqref{eq:DC_Toland_dual}. (Recall that $K^* \partial h\p{Kx} \subseteq \partial \p{h \circ K}\p{x}$ and 
$K \partial \p{g + \varphi}^* \p{K^*y} \subseteq \partial \p{\p{g + \varphi} \circ K^*}\p{y}$.)

Thus, if $\p{\bar x, \bar y} \in \Hh \times \Gg$ is a critical point of the objective function $\Phi$, then $\bar x$ is a critical point of $g + \varphi - h \circ K$ and $\bar y$ is a critical point of 
$h^* - \p{g + \varphi}^* \circ K^*$. 
\end{remark}

\section{The algorithm}\label{sec:algorithm}
Let $(x_0, y_0) \in \Hh \times \Gg$, and let $\p{\gamma_n}_{n\geq 0}$ and $\p{\mu_n}_{n\geq 0}$ be sequences of positive numbers. We propose the following iterative scheme: For all $n\geq 0$ set
\begin{align}
  x_{n+1} &= \Prox_{\gamma_n g} \p{x_n + \gamma_n K^* y_n - \gamma_n \nabla \varphi\p{x_n}}, \label{eq:FBDC_alg:1}\\
  y_{n+1} &= \Prox_{\mu_n h^*} \p{y_n + \mu_n K x_{n+1}}. \label{eq:FBDC_alg:2}
\end{align}

By the inequalities for the proximal points, we have, for every $x, y \in \Hh$ and $n\geq 0$,
\begin{align*}
  g\p{x_{n+1}} - g\p{x} &\leq \frac{1}{\gamma_n} \inpr{x_n + \gamma_n K^* y_n - \gamma_n \nabla \varphi\p{x_n} - x_{n+1}}{x_{n+1} - x} \\
  &= \frac{1}{\gamma_n} \inpr{x_n - x_{n+1}}{x_{n+1} - x} + \inpr{K^* y_n}{x_{n+1} - x} - \!\inpr{\nabla \varphi\p{x_n}}{x_{n+1} - x}, \\
  h^*\p{y_{n+1}} - h^*\p{y} &\leq \frac{1}{\mu_n} \inpr{y_n + \mu_n K x_{n+1} - y_{n+1}}{y_{n+1} - y} \\
  &= \frac{1}{\mu_n} \inpr{y_n - y_{n+1}}{y_{n+1} - y} + \inpr{K x_{n+1}}{y_{n+1} - y}.
\end{align*}
Moreover, using \cite[Theorem 18.15 (iii)]{BauschkeCombettes:2011} and the subdifferential inequality, we have for every $x \in \Hh$ and $n\geq 0$,
\begin{align*}
  \varphi\p{x_{n+1}} - \varphi\p{x_n} &\leq \inpr{\nabla \varphi\p{x_n}}{x_{n+1} - x_n} + \frac{1}{2\beta} \norm{x_n - x_{n+1}}^2, \\
  \varphi\p{x_n} - \varphi\p{x} &\leq \inpr{\nabla \varphi\p{x_n}}{x_n - x}.
\end{align*}
We consider the auxiliary function $\Phi: \Hh \times \Gg \to \RRc$ defined by
\begin{align*}
  \Phi\p{x, y} = g\p{x} + \varphi\p{x} + h^*\p{y} - \inpr{y}{Kx}.
\end{align*}
By the inequalities above, we have, for arbitrary $x\in \Hh$, $y \in \Gg$ and $n\geq 0$,
\begin{align}
  &\norel \Phi\p{x_{n+1}, y_{n+1}} - \Phi\p{x, y} \nonumber \\
  &= g\p{x_{n+1}} - g\p{x} + \varphi\p{x_{n+1}} - \varphi\p{x} + h^*\p{y_{n+1}} - h^*\p{y} + \inpr{y}{Kx} - \inpr{y_{n+1}}{Kx_{n+1}} \nonumber \\
  &\leq \frac{1}{\gamma_n} \inpr{x_n - x_{n+1}}{x_{n+1} - x} + \frac{1}{\mu_n}\inpr{y_n - y_{n+1}}{y_{n+1} - y} + \frac{1}{2\beta} \norm{x_n - x_{n+1}}^2 \nonumber \\
  &\qquad \mathop{+} \inpr{K\p{x - x_{n+1}}}{y - y_n}. \label{eq:FBCD_pdbound}
\end{align}
Furthermore, for any $n \geq 0$,
\begin{align}
  \Phi\p{x_{n+1}, y_n} - \Phi\p{x_n, y_n} &= g\p{x_{n+1}} + \varphi\p{x_{n+1}} - g\p{x_n} - \varphi\p{x_n} + \! \inpr{K^* y_n}{x_n - x_{n+1}} \nonumber \\
  &\leq \p{\frac{1}{2\beta} - \frac{1}{\gamma_n}} \norm{x_n - x_{n+1}}^2, \label{eq:est_Phi_norm_x} \\
  \Phi\p{x_{n+1}, y_{n+1}} - \Phi\p{x_{n+1}, y_n} &= h^*\p{y_{n+1}} - h^*\p{y_n} + \inpr{y_n - y_{n+1}}{K x_{n+1}} \nonumber \\
  &\leq -\frac{1}{\mu_n} \norm{y_n - y_{n+1}}^2. \label{eq:est_Phi_norm_y}
\end{align}
The last two inequalities give rise to the following statement.

\begin{proposition}\label{prop:FBCD_monotonicity}
  For each $n\geq 0$, we have
  \[
    \Phi\p{x_{n+1}, y_{n+1}} \leq \Phi\p{x_{n+1}, y_n} \leq \Phi\p{x_n, y_n},
  \]
  provided that $0 < \gamma_n \leq 2\beta$.
\end{proposition}

\begin{proposition}\label{prop:FBCD_summability}
  Let
  \begin{equation}\label{eq:stepsize_conditions}
    0 < \inf_{n_\geq 0} \gamma_n \leq \sup_{n\geq 0} \gamma_n < 2\beta \qquad \text{and} \qquad 0 < \inf_{n\geq 0} \mu_n \leq \sup_{n\geq 0} \mu_n < +\infty.
  \end{equation}
  Furthermore, let $\inf\setcond{g\p{x} + \varphi\p{x} - h\p{Kx}}{x\in \Hh} > -\infty$. Then,
  \[
    \sum_{n \geq 0} \norm{x_n - x_{n+1}}^2 < +\infty \qquad \text{and} \qquad \sum_{n\geq 0} \norm{y_n - y_{n+1}}^2 < +\infty.
  \]
\end{proposition}
\begin{proof}
  Let $N \geq 1$ be an integer. Sum up \eqref{eq:est_Phi_norm_x} and \eqref{eq:est_Phi_norm_y} for $n = 0, \ldots, N-1$ and obtain
  \[
    \Phi\p{x_N, y_N} - \Phi\p{x_0, y_0} \leq \sum_{n = 0}^{N-1} \p{\frac{1}{2\beta} - \frac{1}{\gamma_n}} \norm{x_n - x_{n+1}}^2 - \sum_{n = 0}^{N-1} \frac{1}{\mu_n} \norm{y_n - y_{n+1}}^2.
  \]
  By assumption, the expression on the left-hand side is bounded below by a fixed real number $M$ for any $N \geq 1$, and so is the right-hand side. The numbers $\p{\frac{1}{\gamma_n} - \frac{1}{2\beta}}$ and $\frac{1}{\mu_n}$ are bounded below by a positive number, say $\varepsilon > 0$, so
  \[
    \sum_{n = 0}^{N-1} \norm{x_n - x_{n+1}}^2 + \sum_{n = 0}^{N-1} \norm{y_n - y_{n+1}}^2 \leq -\frac{M}{\varepsilon}.
  \]
  Since $N$ is arbitrary, the series converge.
\end{proof}
\begin{proposition}\label{prop:FBCD_clusterponts}
  Let $\inf\setcond{g\p{x} + \varphi\p{x} - h\p{Kx}}{x\in \Hh} > -\infty$ and \eqref{eq:stepsize_conditions} be satisfied. If $\p{x_n}_{n\geq 0}$ and $\p{y_n}_{n\geq 0}$ are bounded, then 
  \begin{enumerate}
    \item every cluster point of $\p{x_n}_{n\geq 0}$ is a critical point of \eqref{eq:DC_Problem},
    \item every cluster point of $\p{y_n}_{n\geq 0}$ is a critical point of \eqref{eq:DC_Toland_dual} and
    \item every cluster point of $\p{x_n, y_n}_{n\geq 0}$ is a critical point of \eqref{eq:DC_pd}.
  \end{enumerate}
\end{proposition}
\begin{proof}
  Let $\bar x$ be a cluster point of $\p{x_n}_{n\geq 0}$. Let $\p{x_{n_k}}_{k\geq 0}$ be a subsequence of $\p{x_n}_{n\geq 0}$ such that $x_{n_k} \to \bar x$. By another transition to a subsequence, we can guarantee $y_{n_k} \to \bar y$ for some $\bar y \in \Hh$, since $\p{y_{n_k}}_{k\geq 0}$ is bounded. By \eqref{eq:FBDC_alg:1} and \eqref{eq:FBDC_alg:2}, we obtain, for every $k\geq 0$,
  \begin{align*}
    \frac{x_{n_k} - x_{n_k+1}}{\gamma_{n_k}} + K^* y_{n_k} - \nabla \varphi\p{x_{n_k}} &\in \partial g\p{x_{n_k+1}} \\
    \text{and } \frac{y_{n_k} - y_{n_k + 1}}{\mu_{n_k}} + K x_{n_k+1} &\in \partial h^*\p{y_{n_k + 1}},
  \end{align*}
  respectively. By Proposition \ref{prop:FBCD_summability}, the first summands on the left-hand side of the above inclusions tend to zero as $k\to \infty$. Using the continuity of $\nabla \varphi$ and the closedness of the graphs of $\partial g$ and $\partial h^*$ and passing to the limit, 
  we get $K^* \bar y - \nabla \varphi\p{\bar x} \in \partial g\p{\bar x}$ and $K \bar x \in \partial h^*\p{\bar y}$, which means that $\p{\bar x, \bar y}$ is a critical point of $\Phi$. The first statement follows by considering Remark \ref{rem:critical}. For the second statement, 
  one has to choose $\bar x$ and $\bar y$ in reverse order, for the third one, they are chosen at the same time.
\end{proof}
\begin{remark}
  It is clear that one cannot expect the cluster points to be minima, since it is easy to see that $\p{\bar x, \bar y}$ is a fixed point of the iteration \eqref{eq:FBDC_alg:1}--\eqref{eq:FBDC_alg:2} if and only if \eqref{eq:stat_incl1} and \eqref{eq:stat_incl2} are satisfied, i.e., 
  if and only if $\p{\bar x, \bar y}$ is a critical point for $\Phi$ (independent of the choice of the parameters $\p{\gamma_n}_{n\geq 0}$ and $\p{\mu_n}_{n\geq 0}$).
\end{remark}
\begin{proposition}\label{prop:fixed_points}
  Let \eqref{eq:stepsize_conditions} be satisfied. For any $n\geq 0$, the following statements are equivalent:
  \begin{enumerate}
    \item $\p{x_n, y_n}$ is a critical point of $\Phi$;
    \item $\p{x_{n+1}, y_{n+1}} = \p{x_n, y_n}$;
    \item $\Phi\p{x_{n+1}, y_{n+1}} = \Phi\p{x_n, y_n}$.
  \end{enumerate}
\end{proposition}
\begin{proof}
  It is easily seen by the formula \eqref{eq:prox_subdiff_inclusion} that the first two items are equivalent. The equivalence of the latter two items follows by \eqref{eq:est_Phi_norm_x} and \eqref{eq:est_Phi_norm_y}.
\end{proof}

Next, we summarise the convergence properties of the prox-prox algorithm. To this end, we denote by $\omega\p{x_0, y_0}$ the set of cluster points of the iteration generated by \eqref{eq:FBDC_alg:1} and \eqref{eq:FBDC_alg:2} with the initial points $x_0$ and $y_0$. See also \cite[Lemma 5]{BolteSabachTeboulle:2014} for an analogous result for a nonconvex forward-backward scheme.
\begin{lemma}\label{lem:FBDC_properties}
 Let $\Hh$ and $\Gg$ be two real finite-dimensional Hilbert spaces, let $g: \Hh \to \RRc$ and $h: \Gg \to \RRc$ be proper, convex and lower semicontinuous functions, let $\varphi: \Hh \to \RR$ be a convex, Fr\'echet differentiable function with a $\frac{1}{\beta}$-Lipschitz continuous gradient, for some $\beta >0$, 
 and let $K: \Hh \to \Gg$ be a linear mapping. Let the sequences $\p{\gamma_n}_{n\geq 0}$ and $\p{\mu_n}_{n\geq 0}$ satisfy \eqref{eq:stepsize_conditions}. Moreover, 
 assume that the sequence $\p{x_n, y_n}_{n \geq 0}$ generated by \eqref{eq:FBDC_alg:1} and \eqref{eq:FBDC_alg:2} is bounded. Then the following assertions hold:
  \begin{enumerate}
    \item $\emptyset \neq \omega\p{x_0, y_0} \subseteq \crit \Phi \subseteq \crit \p{g + \varphi - h \circ K} \times \crit\p{h^* - \p{g + \varphi}^* \circ K^*}$,
    \item $\lim_{n \to \infty} \dist\p{\p{x_n, y_n}, \omega\p{x_0, y_0}} = 0$,
    \item if the common optimal value of the problems \eqref{eq:DC_Problem}, \eqref{eq:DC_Toland_dual} and \eqref{eq:DC_pd} is $> -\infty$, then $\omega\p{x_0, y_0}$ is a nonempty, compact and connected set, and so are the sets of the limit points of the sequences $\p{x_n}_{n\geq 0}$ and $\p{y_n}_{n\geq 0}$,
    \item \label{item:lem:FBDC_properties_const} the objective function $\Phi$ is finite and constant on $\omega\p{x_0, y_0}$ provided that the optimal value is finite.
  \end{enumerate}
\end{lemma}
\begin{proof}
  \begin{enumerate}
    \item It is clear that the set of cluster points of a bounded sequence is nonempty. That every cluster point is critical for $\Phi$, is the statement of Proposition \ref{prop:FBCD_clusterponts}. The last inclusion is discussed in Remark \ref{rem:critical}.
    \item Assume that the assertion does not hold. In this case, there exists an $\varepsilon > 0$ and a subsequence $\p{x_{n_k}, y_{n_k}}_{k\geq 0}$ of $\p{x_n, y_n}_{n\geq 0}$ with $\dist\p{\p{x_{n_k}, y_{n_k}}, \omega\p{x_0, y_0}} > \varepsilon$ for all $k\geq 0$. The subsequence is bounded, so it has a cluster point, which is a cluster point of the original sequence $\p{x_n, y_n}_{n\geq 0}$ as well, thus an element of $\omega\p{x_0, y_0}$. This contradicts the assumption $\dist\p{\p{x_{n_k}, y_{n_k}}, \omega\p{x_0, y_0}} > \varepsilon$ for all $k\geq 0$.
    \item Since the sequence $\p{x_n, y_n}_{n\geq 0}$ is bounded, the sets
      \[
        \Omega_k \defeq \cl\p{\bigcup_{n\geq k} \set{\p{x_n, y_n}}}
      \]
      are bounded and closed, hence compact for any $k\geq 0$. Their intersection $\bigcap_{n\geq 0} \Omega_n$, which equals the set of cluster points of $\p{x_n, y_n}_{n\geq 0}$, is therefore compact, too. The connectedness follows from Proposition \ref{prop:FBCD_summability}. See the proof of \cite[Lemma 5 (iii)]{BolteSabachTeboulle:2014} for the details.

    \item According to Proposition \ref{prop:FBCD_monotonicity}, the function values $\Phi\p{x_n, y_n}$ are monotonically decreasing, thus convergent, say $\Phi\p{x_n, y_n} \to \ell$. Let $\p{\bar x, \bar y}$ be an arbitrary limit point of the sequence $\p{x_n, y_n}_{n\geq 0}$, and let $\p{x_{n_k}, y_{n_k}}_{k\geq 0}$ be a subsequence converging to $\p{\bar x, \bar y}$ as $k\to \infty$. By lower semicontinuity, we have $\Phi\p{\bar x, \bar y} \leq \lim_{k\to \infty} \Phi\p{x_{n_k}, y_{n_k}} = \ell$. On the other hand, consider \eqref{eq:FBCD_pdbound} with $x = \bar x$ and $y = \bar y$. The right-hand side converges to $0$ as we let $n\to \infty$ along the subsequence $\p{n_k}_{k \geq 0}$, so $\ell = \lim_{n\to \infty} \Phi\p{x_n, y_n} \leq \Phi\p{\bar x, \bar y}$. \qedhere
  \end{enumerate}
\end{proof}
\begin{remark}
  To guarantee the boundedness of the iterates, one could assume that the objective function of the primal-dual minimization problem \eqref{eq:DC_pd} is coercive, i.e., the lower level sets are bounded.
\end{remark}

\section{Convergence under Kurdyka--\L ojasiewicz assumptions}\label{sec:KL}

In the next step, we shall assume the Kurdyka--\L ojasiewicz property for the functions involved. Let us recall the definition and some basic properties. By $\Theta_\eta$, for $\eta \in \ocint{0}{+\infty}$, we denote the set of all concave and continuous functions $\vartheta: \coint{0}{\eta} \to \RR$ with the following properties:
\begin{enumerate}
  \item $\vartheta\p{0} = 0$,
  \item $\vartheta$ is continuously differentiable on $\opint{0}{\eta}$ and continuous at $0$,
  \item $\vartheta'\p{s} > 0$ for all $s\in \opint{0}{\eta}$.
\end{enumerate}

\begin{definition}
  Let $\Hh$ be a real finite-dimensional Hilbert space, and let $\Phi: \Hh \to \RRc$ be a proper and lower semicontinuous function. We say that $\Phi$ satisfies the \emph{Kurdyka--\L ojasiewicz property} at $\bar x \in \dom \partial_L \Phi \defeq \setcond{x\in \Hh}{\partial_L \Phi\p{x} \neq \emptyset}$ if there exists some $\eta \in \ocint{0}{+\infty}$, a neighbourhood $U$ of $\bar x$ and a function $\vartheta \in \Theta_\eta$ such that for all
  \[
    x \in U \cap \setcond{x\in \Hh}{\Phi\p{\bar x} < \Phi\p{x} < \Phi\p{\bar x} + \eta}
  \]
  the following inequality holds
  \[
    \vartheta'\p{\Phi\p{x} - \Phi\p{\bar x}} \cdot \dist\p{0, \partial_L \Phi \p{x}} \geq 1.
  \]
  We call $\Phi$ a \emph{K\L{} function} if it satisfies the Kurdyka--\L ojasiewicz property at each point $\bar x \in \dom \partial_L \Phi$.
\end{definition}

The following uniform K\L{} property is according to \cite[Lemma 6]{BolteSabachTeboulle:2014}.
\begin{lemma}\label{lem:uniform_KL}
  Let $\Omega$ be a compact set, and let $\Phi: \Hh \to \RRc$ be a proper and lower semicontinuous function. Assume that $\Phi$ is constant on $\Omega$ and satisfies the K\L{} property at each point of $\Omega$. Then there exist $\varepsilon > 0$, $\eta > 0$ and $\vartheta \in \Theta_\eta$ such that for all $\bar u \in \Omega$ and all $u$ in the intersection
  \begin{equation}\label{eq:uniform_KL_intersection}
    \setcond{u \in \Hh}{\dist\p{u, \Omega} < \varepsilon} \cap \setcond{u\in \Hh}{\Phi\p{\bar u} < \Phi\p{u} < \Phi\p{\bar u} + \eta}
  \end{equation}
  one has
  \[
    \vartheta'\p{\Phi\p{u} - \Phi\p{\bar u}} \cdot \dist\p{0, \partial_L \Phi\p{u}} \geq 1.
  \]
\end{lemma}

In the K\L{} property, we need the distance of a subgradient from zero. In our algorithm, we have the following result.

\begin{lemma}\label{lem:subgradient_estimation}
  For each $n\geq 1$ with $\gamma_{n-1} < 2\beta$, there exist $\p{x^*_n, y^*_n} \in \Hh \times \Gg$ with $\p{x^*_n, y^*_n} \in \partial_L \Phi\p{x_n, y_n}$ and
  \begin{align*}
    \norm{x^*_n} &\leq \norm{K} \norm{y_{n-1} - y_n} + \frac{1}{\gamma_{n-1}} \norm{x_{n-1} - x_n}, \\
    \norm{y^*_n} &\leq \frac{1}{\mu_{n-1}} \norm{y_{n-1} - y_n}.
  \end{align*}
\end{lemma}
\begin{proof}
  From the definition of the algorithm, we have, for each $n \geq 1$,
  \begin{align*}
    \frac{x_{n-1} - x_n}{\gamma_{n-1}} + K^* y_{n-1} - \nabla \varphi\p{x_{n-1}} &\in \partial g\p{x_n}, \\
    \frac{y_{n-1} - y_n}{\mu_{n-1}} + K x_n &\in \partial h^*\p{y_n}.
  \end{align*}
Consider the function $\widetilde \Phi: \Hh \times \Gg \to \RRc, \widetilde \Phi\p{x, y} \defeq g\p{x} + \varphi\p{x} + h^*\p{y}$. By the usual calculus of the convex subdifferential and \cite[Proposition 8.12]{RockafellarWets:1998}, for each $n \geq 1$
  \[
    \partial_L \widetilde \Phi\p{x_n, y_n} = \p{\partial g\p{x_n} + \nabla \varphi\p{x_n}} \times \partial h^*\p{y_n}.
  \]
  By \cite[Exercise 8.8]{RockafellarWets:1998}, we have for each $n \geq 1$
  \begin{align}\label{eq:Phi_inclusion}
    \partial_L \Phi\p{x_n, y_n} &= \partial_L \widetilde\Phi\p{x_n, y_n} - \p{K^* y_n, Kx_n} \nonumber \\
    &= \p{\partial g\p{x_n} + \nabla \varphi\p{x_n} - K^* y_n} \times \p{\partial h^*\p{y_n} - Kx_n},
  \end{align}
  thus,
  \begin{align*}
    \begin{pmatrix}
      x_n^* \\
      y_n^*
    \end{pmatrix} \defeq
    \begin{pmatrix}
      \frac{x_{n-1} - x_n}{\gamma_{n-1}} + \nabla \varphi\p{x_n} - \nabla \varphi\p{x_{n-1}} + K^*\p{y_{n-1} - y_n} \\
      \frac{y_{n-1} - y_n}{\mu_{n-1}}
    \end{pmatrix}\in \partial_L \Phi\p{x_n, y_n}.
  \end{align*}
  Now, we estimate for each $n \geq 1$
  \begin{align*}
    \norm{x_n^*} &\leq \norm{K} \norm{y_{n-1} - y_n} + \frac{1}{\gamma_{n-1}}\norm{\p{\Id - \gamma_{n-1} \nabla \varphi}\p{x_{n-1}} - \p{\Id - \gamma_{n-1} \nabla \varphi}\p{x_n}}.
  \end{align*}
  By the Baillon--Haddad theorem \cite[Corollary 18.16]{BauschkeCombettes:2011}, $\nabla \varphi$ is $\beta$-cocoercive. By \cite[Proposition 4.33]{BauschkeCombettes:2011}, $\Id - \gamma_{n-1} \nabla \varphi$ is nonexpansive for $\gamma_{n-1} < 2\beta$, which leads to the desired conclusion.
\end{proof}

\subsection{The case when $\Phi$ is a K\L{} function}
\begin{theorem}\label{th1}
  Let
  \begin{align*}
    0 < \underline \gamma \defeq \inf_{n\geq 0} \gamma_n &\leq \overline \gamma \defeq \sup_{n\geq 0} \gamma_n < \beta, \\
    0 < \underline \mu \defeq \inf_{n\geq 0} \mu_n &\leq \overline \mu \defeq \sup_{n\geq 0} \mu_n < +\infty.
  \end{align*}
  Suppose that $\Phi$ is in addition a K\L{} function bounded from below. Then $\p{x_n, y_n}_{n\geq 0}$ is a Cauchy sequence, thus convergent to a critical point of $\Phi$.
\end{theorem}
\begin{proof}
  Let $\Omega \defeq \omega\p{x_0, y_0}$, and let $\ell \in \RR$ be the value of $\Phi$ on $\Omega$ (see item \ref{item:lem:FBDC_properties_const} of Lemma \ref{lem:FBDC_properties}). If $\Phi\p{x_n, y_n} = \ell$ for some $n \geq 0$, then, by \eqref{eq:est_Phi_norm_x} and \eqref{eq:est_Phi_norm_y}, $x_{n+1} = x_n$ and $y_{n+1} = y_n$, and the assertion holds. Therefore, we assume $\Phi\p{x_n, y_n} > \ell$ for all $n\geq 0$.
  
  Let $\varepsilon > 0$, $\eta > 0$ and $\vartheta \in \Theta_\eta$ be as provided by Lemma \ref{lem:uniform_KL}.
  Since $\Phi\p{x_n, y_n} \to \ell$ as $n\to +\infty$, we find $n_1 \geq 0$ with $\Phi\p{x_n, y_n} < \ell + \eta$ for $n\geq n_1$. Since $\dist\p{\p{x_n, y_n}, \Omega} \to 0$ as $n\to +\infty$, we find $n_2 \geq 0$ with $\dist\p{\p{x_n, y_n}, \Omega} < \varepsilon$ for $n\geq n_2$. 
  
  In the following, fix an arbitrary $n\geq n_0 \defeq \max\set{n_1, n_2, 1}$. Then $\p{x_n, y_n}$ is an element of the intersection \eqref{eq:uniform_KL_intersection}. Consequently,
  \begin{equation}\label{eq:KL_derivatives}
    \vartheta'\p{\Phi\p{x_n, y_n} - \Phi\p{\bar x, \bar y}} \cdot \dist\p{\p{0, 0}, \partial_L \Phi\p{x_n, y_n}} \geq 1.
  \end{equation}
  By the concavity of $\vartheta$, we get, for all $s \in \p{0, \eta}$,
  \[
    \vartheta\p{s} - \vartheta\p{\Phi\p{x_n, y_n} - \Phi\p{\bar x, \bar y}} \leq \vartheta'\p{\Phi\p{x_n, y_n} - \Phi\p{\bar x, \bar y}} \cdot \p{s - \Phi\p{x_n, y_n} + \Phi\p{\bar x, \bar y}},
  \]
  so, setting in particular $s := \Phi\p{x_{n+1}, y_{n+1}} - \Phi\p{\bar x, \bar y} \in \opint{0}{\eta}$,
  \begin{align*}
    &\norel \p{\vartheta\p{\Phi\p{x_n, y_n} - \Phi\p{\bar x, \bar y}} - \vartheta\p{\Phi\p{x_{n+1}, y_{n+1}} - \Phi\p{\bar x, \bar y}}} \cdot \norm{\p{x_n^*, y_n^*}} \\
     &\geq \vartheta'\p{\Phi\p{x_n, y_n} - \Phi\p{\bar x, \bar y}} \cdot \p{\Phi\p{x_n, y_n} - \Phi\p{x_{n+1}, y_{n+1}}} \cdot \norm{\p{x_n^*, y_n^*}} \\
     &\geq \vartheta'\p{\Phi\p{x_n, y_n} - \Phi\p{\bar x, \bar y}} \cdot \p{\Phi\p{x_n, y_n} - \Phi\p{x_{n+1}, y_{n+1}}} \cdot \dist\p{\p{0, 0}, \partial_L \Phi\p{x_n, y_n}} \\
     &\geq \p{\Phi\p{x_n, y_n} - \Phi\p{x_{n+1}, y_{n+1}}}.
  \end{align*}
%  \begin{align*}
%    1 &\leq \vartheta'\p{\Phi\p{x_n, y_n} - \Phi\p{\bar x, \bar y}} \cdot \dist\p{\p{0, 0}, \partial_L \Phi\p{x_n, y_n}} \\
%    &\leq \frac{\vartheta\p{\Phi\p{x_n, y_n} - \Phi\p{\bar x, \bar y}} - \vartheta\p{\Phi\p{x_{n+1}, y_{n+1}} - \Phi\p{\bar x, \bar y}}}{\Phi\p{x_n, y_n} - \Phi\p{x_{n+1}, y_{n+1}}} \cdot \dist\p{\p{0, 0}, \partial_L \Phi\p{x_n, y_n}} \\
%    &\leq \frac{\vartheta\p{\Phi\p{x_n, y_n} - \Phi\p{\bar x, \bar y}} - \vartheta\p{\Phi\p{x_{n+1}, y_{n+1}} - \Phi\p{\bar x, \bar y}}}{\Phi\p{x_n, y_n} - \Phi\p{x_{n+1}, y_{n+1}}} \cdot \norm{\p{x_n^*, y_n^*}}
%  \end{align*}
  Moreover, by \eqref{eq:est_Phi_norm_x} and \eqref{eq:est_Phi_norm_y},
  \[
    \Phi\p{x_n, y_n} - \Phi\p{x_{n+1}, y_{n+1}} \geq \p{\frac{1}{\gamma_n} - \frac{1}{2\beta}} \norm{x_n - x_{n+1}}^2 + \frac{1}{\mu_n} \norm{y_n - y_{n+1}}^2.
  \]
  Let us define the following shorthands
  \begin{align*}
    \delta_n &\defeq \sqrt{\p{\frac{1}{\gamma_n} - \frac{1}{2\beta}} \norm{x_n - x_{n+1}}^2 + \frac{1}{\mu_n} \norm{y_n - y_{n+1}}^2}, \\
    \varepsilon_n &\defeq \vartheta\p{\Phi\p{x_n, y_n} - \Phi\p{\bar x, \bar y}}
  \end{align*}
  for $n\geq n_0$ to obtain the inequality
  \[
    \p{\varepsilon_n - \varepsilon_{n+1}} \cdot \norm{\p{x_n^*, y_n^*}} \geq \delta_n^2.
  \]
%  \[
%    1 \leq \frac{\varepsilon_n - \varepsilon_{n+1}}{\delta_n^2} \cdot \norm{\p{x_n^*, y_n^*}}.
%  \]
  By the arithmetic-geometric inequality, for any $r > 0$ and $n\geq n_0$
  \begin{align}
    \delta_n &\leq \sqrt{\p{r\norm{\p{x_n^*, y_n^*}}} \cdot \p{\frac{1}{r}\p{\varepsilon_n - \varepsilon_{n+1}}}} \nonumber \\
    & \leq \frac{1}{2} \p{r\norm{\p{x_n^*, y_n^*}} + \frac{1}{r} \p{\varepsilon_n - \varepsilon_{n+1}}} \nonumber \\
& \leq r\norm{\p{x_n^*, y_n^*}} + \frac{1}{r} \p{\varepsilon_n - \varepsilon_{n+1}} \label{eq:estimation_delta_norm}
  \end{align}
%  Now fix $r > 0$ and $n\geq 1$ and assume that $\delta_n \geq r \norm{\p{x_n^*, y_n^*}}$. In this case,
%  \[
%    1 \leq \frac{\varepsilon_n - \varepsilon_{n+1}}{r \delta_n},
%  \]
%  so $\delta_n \leq \frac{1}{r} \p{\varepsilon_n - \varepsilon_{n+1}}$. Thus, for all $r > 0$,
%  \begin{align}
%    \delta_n &\leq \max\set{r \norm{\p{x_n^*, y_n^*}}, \frac{1}{r} \p{\varepsilon_n - \varepsilon_{n+1}}} \nonumber \\
%    &\leq r\norm{\p{x_n^*, y_n^*}} + \frac{1}{r} \p{\varepsilon_n - \varepsilon_{n+1}} \label{eq:estimation_delta_norm}
%  \end{align}
  (recall that, by Proposition \ref{prop:FBCD_monotonicity} and the properties of $\vartheta$, the sequence $\p{\varepsilon_n}_{n\geq n_0}$ is decreasing, so $\varepsilon_n - \varepsilon_{n+1} \geq 0$). On the other hand, by Lemma \ref{lem:subgradient_estimation} and the inequality $2ab \leq a^2 + b^2$ ($a, b \geq 0$), for any $n\geq n_0$
  \begin{align}
    \norm{\p{x_n^*, y_n^*}}^2 \leq & \p{\norm{K}^2 + \frac{1}{\mu_{n-1}^2}} \norm{y_{n-1} - y_n}^2 + \frac{1}{\gamma_{n-1}^2} \norm{x_{n-1} - x_n}^2 +\nonumber \\
& + \frac{2\norm{K}}{\gamma_{n-1}} \norm{x_{n-1} - x_n}\norm{y_{n-1} - y_n} \nonumber \\
\leq & \p{2\norm{K}^2 + \frac{1}{\mu_{n-1}^2}} \norm{y_{n-1} - y_n}^2 + \frac{2}{\gamma_{n-1}^2} \norm{x_{n-1} - x_n}^2 \nonumber \\
 \leq & C^2_n \delta^2_{n-1}, \label{eq:estimation_norm_delta}
  \end{align}
  with
\begin{align*}
 C_n \defeq & \sqrt{\max\set{\frac{\frac{2}{\gamma_{n-1}^2}}{\frac{1}{\gamma_{n-1}} - \frac{1}{2\beta}}, \frac{2\norm{K}^2 + \frac{1}{\mu_{n-1}^2}}{\frac{1}{\mu_{n-1}}}}}\\
= & \sqrt{\max\set{\frac{4\beta}{\gamma_{n-1}\p{2\beta - \gamma_{n-1}}}, \frac{1 + 2\norm{K}^2 \mu_{n-1}^2}{\mu_{n-1}}}}.
  \end{align*}
For all $n\geq n_0$,
  \[
    C_n \leq C_0 \defeq \sqrt{\max\set{\frac{4\beta}{\underline \gamma\p{2\beta - \overline \gamma}}, \frac{1 + 2\norm{K}^2 \overline \mu^2}{\underline \mu}}}.
  \]
  Combined with \eqref{eq:estimation_delta_norm}, we obtain
  \begin{equation}\label{eq:estimation_delta}
    \delta_n \leq rC_0 \delta_{n-1} + \frac{1}{r} \p{\varepsilon_n - \varepsilon_{n+1}}.
  \end{equation}
  For any $k \geq n_0+1$, we have, by iteration,
  \[
    \delta_k \leq \p{rC_0}^{k - n_0} \delta_{n_0} + \sum_{n = 0}^{k - n_0 - 1} \frac{\p{rC_0}^n}{r} \p{\varepsilon_{k - n} - \varepsilon_{k - n + 1}},
  \]
  therefore, for any $N \geq n_0+1$ and $0 < r < \frac{1}{C_0}$,
  \begin{align*}
    \sum_{k = n_0+1}^N \delta_k &\leq \sum_{k = n_0+1}^N \p{\p{rC_0}^{k - n_0} \delta_{n_0} + \sum_{n = 0}^{k - n_0 - 1} \frac{\p{rC_0}^n}{r} \p{\varepsilon_{k - n} - \varepsilon_{k - n + 1}}} \\
    &= \sum_{k = 0}^{N - n_0-1} \p{rC_0}^{k+1} \delta_{n_0} + \sum_{k = 0}^{N - n_0-1} \sum_{n = 0}^{k} \frac{\p{rC_0}^n}{r} \p{\varepsilon_{k + n_0 - n+1} - \varepsilon_{k + n_0 - n + 2}} \\
    &\leq \frac{rC_0\delta_{n_0}}{1 - rC_0} + \sum_{n = 0}^{N - n_0 - 1} \frac{\p{rC_0}^n}{r} \sum_{k = n}^{N - n_0-1} \p{\varepsilon_{k + n_0 - n+1} - \varepsilon_{k + n_0 - n + 2}} \\
    &\leq \frac{rC_0\delta_{n_0}}{1 - rC_0} + \sum_{n = 0}^{N - n_0 - 1} \frac{\p{rC_0}^n}{r} \varepsilon_{n_0 + 1} \\
    &\leq \frac{rC_0\delta_{n_0}}{1 - rC_0} + \frac{\varepsilon_{n_0 + 1}}{r\p{1 - rC_0}}. 
  \end{align*}
  The right-hand side does not depend on $N$, thus, we conclude that $\sum_{k = n_0+1}^\infty \delta_k$ is finite, and so are $\sum_{k = n_0+1}^\infty \norm{x_n - x_{n+1}}$ and $\sum_{k = n_0+1}^{\infty} \norm{y_n - y_{n+1}}$.
\end{proof}

\subsection{Convergence rates}

\begin{lemma}
  Assume that $\Phi$ is a K\L{} function with $\vartheta\p{t} = Mt^{1 - \theta}$ for some $M > 0$ and $0 \leq \theta < 1$. Let $\bar x$ and $\bar y$ the limit points of the sequences $\p{x_n}_{n\geq 0}$ and $\p{y_n}_{n\geq 0}$, respectively (which exist due to Theorem \ref{th1}). Then the following convergence rates are guaranteed:
  \begin{enumerate}
    \item if $\theta = 0$, then there exists $n_0 \geq 0$, such that $x_n = x_{n_0}$ and $y_n = y_{n_0}$ for $n\geq n_0$;
    \item if $0 < \theta \leq \frac{1}{2}$, then there exist $c > 0$ and $0 \leq q < 1$ such that
      \[
        \norm{x_n - \bar x} \leq cq^n \qquad \text{and}\qquad \norm{y_n - \bar y} \leq cq^n
      \]
      for all $n \geq 0$;
    \item if $\frac{1}{2} < \theta < 1$, then there exists $c > 0$ such that
      \[
        \norm{x_n - \bar x} \leq c n^{-\frac{1-\theta}{2\theta - 1}} \qquad \text{and}\qquad \norm{y_n - \bar y} \leq c n^{-\frac{1 - \theta}{2\theta - 1}}
      \]
      for all $n \geq 0$.
  \end{enumerate}
\end{lemma}
\begin{proof}
  \begin{enumerate}
    \item First, let $\theta = 0$. Assume to the contrary (see Proposition \ref{prop:fixed_points}) that for any $n \geq 0$, $\p{x_{n+1}, y_{n+1}} \neq \p{x_n, y_n}$. We have $\vartheta'\p{t} = M$ for all $t > 0$ and thus, by \eqref{eq:KL_derivatives},
      \[
        M \cdot \norm{\p{x_n^*, y_n^*}} \geq 1 \ \mbox{for any} \ n \geq 1,
      \]
      which contradicts either Lemma \ref{lem:subgradient_estimation} or Proposition \ref{prop:FBCD_summability}.
  \end{enumerate}
  Before considering the other cases, assume from now on that $\p{x_n, y_n}$ is not a critical point of $\Phi$ for any $n\geq 0$. Notice that $\vartheta'\p{t} = \p{1 - \theta} M t^{-\theta}$.
  In the proof of Theorem \ref{th1}, we have shown that for $0<r<\frac{1}{C_0}$
  \begin{align*}
    \sum_{k = n_0+1}^\infty \delta_k & \leq \frac{rC_0\delta_{n_0}}{1 - rC_0} + \frac{\varepsilon_{n_0 + 1}}{r\p{1 - rC_0}} \\
    &= \frac{rC_0\delta_{n_0}}{1 - rC_0} + \frac{M\p{\Phi\p{x_{n_0 + 1}, y_{n_0 + 1}} - \Phi\p{\bar x, \bar y}}^{1 - \theta}}{r\p{1 - rC_0}} \\
    &= \frac{rC_0\delta_{n_0}}{1 - rC_0} + \frac{M^{1 + \frac{1 - \theta}{\theta}}\p{1 - \theta}^{\frac{1 - \theta}{\theta}}}{r\p{1 - rC_0}\vartheta'\p{\Phi\p{x_{n_0 + 1}, y_{n_0 + 1}} - \Phi\p{\bar x, \bar y}}^{\frac{1-\theta}{\theta}}} \\
    &\leq \frac{rC_0\delta_{n_0}}{1 - rC_0} + \frac{M^{\frac{1}{\theta}}\p{1 - \theta}^{\frac{1 - \theta}{\theta}}\norm{\p{x_{n_0 + 1}^*, y_{n_0 + 1}^*}}^{\frac{1-\theta}{\theta}}}{r\p{1 - rC_0}},
  \end{align*}
  where the last inequality follows from the K\L{} property (notice that $\Phi\p{x_{n_0 + 1}, y_{n_0 + 1}} - \Phi\p{\bar x, \bar y} > 0$ because we assumed that $\p{x_{n_0 + 1}, y_{n_0 + 1}}$ is not a critical point of $\Phi$). We can repeat this calculation for any $n\geq n_0+1$ instead of $n_0+1$, because such an $n$ would meet the criteria according to which we chose $n_0+1$. Thus, we obtain from \eqref{eq:estimation_norm_delta}, for $n\geq n_0+1$,
 \begin{equation}\label{eq:convergence_rate_basic_estimate2}
    \sum_{k = n + 1}^\infty \delta_k \leq \frac{rC_0\delta_n}{1 - rC_0} + \frac{M^{\frac{1}{\theta}} \p{1 - \theta}^{\frac{1 - \theta}{\theta}} \p{C_0 \delta_n}^{\frac{1 - \theta}{\theta}}}{r\p{1 - rC_0}}.
  \end{equation}
  The rest of the proof follows in the lines of \cite[Theorem 2]{AttouchBolte:2009}:

  \begin{enumerate}\setcounter{enumi}{1}
    \item Let $0 < \theta \leq \frac{1}{2}$. Then $1 \leq \frac{1 - \theta}{\theta} < +\infty$, so $\delta_n \to 0$ as $n\to \infty$ implies that the first term on the right-hand side of \eqref{eq:convergence_rate_basic_estimate2} is the dominant one. Therefore, we find $n_1 \geq n_0+1$ and $C_1 > 0$ such that
      \[
        \sum_{k = n + 1}^{\infty} \delta_k \leq C_1 \delta_n = C_1 \p{\sum_{k = n}^{\infty} \delta_k - \sum_{k = n + 1}^\infty \delta_k}
      \]
      for any $n\geq n_1$. Thus, for any  $n\geq n_1$,
      \[
        \sum_{k = n + 1}^{\infty} \delta_k \leq \frac{C_1}{1 + C_1}  \sum_{k = n}^{\infty} \delta_k.
      \]
      By induction, for any $n\geq n_1+1$,
      \[
        \delta_{n} \leq \sum_{k = n}^\infty \delta_k \leq \p{\frac{C_1}{1 + C_1}}^{n - n_1} \sum_{k = n_1}^\infty \delta_k,
      \]
      which proves the assertion.
    \item Let $\frac{1}{2} < \theta < 1$. Then $0 < \frac{1 - \theta}{\theta} < 1$, so $\delta_n \to 0$ as $n\to \infty$ implies that the second term on the right-hand side of \eqref{eq:convergence_rate_basic_estimate2} is the dominant one. Therefore, we find $n_1 \geq n_0+1$ and $C_1 > 0$ such that
      \[
        \sum_{k = n+1}^\infty \delta_k \leq C_1 \delta_n^{\frac{1 - \theta}{\theta}}
      \]
      for any $n\geq n_1$. Then, for any $n\geq n_1$,
      \begin{align*}
        \p{\sum_{k = n+1}^\infty \delta_k}^{\frac{\theta}{1 - \theta}} \leq C_1^{\frac{\theta}{1 - \theta}} \p{\sum_{k = n}^\infty \delta_k - \sum_{k = n+1}^\infty \delta_k}.
      \end{align*}
      We define $h: \opint{0}{+\infty} \to \RR$, $h\p{s} = s^{-\frac{\theta}{1 - \theta}}$ and notice that $h$ is monotonically decreasing as is the sequence $\p{\sum_{k = n}^\infty \delta_k}_{n \geq n_1}$. Therefore, for any $n\geq n_1$,
      \begin{align*}
        1 &\leq C_1^{\frac{\theta}{1 - \theta}} h\p{\sum_{k = n+1}^\infty \delta_k} \p{\sum_{k = n}^\infty \delta_k - \sum_{k = n + 1}^\infty \delta_k} \\
        &\leq C_1^{\frac{\theta}{1 - \theta}} \int_{\sum_{k = n + 1}^\infty \delta_k}^{\sum_{k = n}^\infty \delta_k} h\p{s} \dd s \\
        &= -C_1^{\frac{\theta}{1 - \theta}} \frac{1 - \theta}{2\theta - 1}\p{ \p{\sum_{k = n}^\infty \delta_k}^{-\frac{2\theta - 1}{1 - \theta}} - \p{\sum_{k = n + 1}^\infty \delta_k}^{-\frac{2\theta - 1}{1 - \theta}}}.
      \end{align*}
      Thus, by induction, for any $n\geq n_1+1$,
      \[
        \p{\sum_{k = n}^\infty \delta_k}^{-\frac{2\theta - 1}{1 - \theta}} + \frac{\p{2\theta - 1}\p{n - n_1}}{C_1^{\frac{\theta}{1 - \theta}}\p{1 - \theta}} \leq \p{\sum_{k = n_1}^\infty \delta_k}^{-\frac{2\theta - 1}{1 - \theta}}.
      \]
      The assertion follows by
      \[
        \delta_n \leq \sum_{k = n}^\infty \delta_k \leq \p{\p{\sum_{k = n_1}^\infty \delta_k}^{-\frac{2\theta - 1}{1 - \theta}} + \frac{\p{2\theta - 1}\p{n - n_1}}{C_1\p{1 - \theta}}}^{-\frac{1 - \theta}{2\theta - 1}}\qquad \text{for any } n \geq n_1+1. \qedhere
      \]
  \end{enumerate}
\end{proof}

\section{Application to image processing}
Consider an image of the size $m \times n$ pixels. (For the sake of simplicity, we consider gray-scale pictures only.) It can be represented by a vector $x \in \Hh \defeq \RR^{mn}$ of size $mn$ with entries in $\clint{0}{1}$ (where $0$ represents pure black and $1$ represents pure white).

The original image $x \in \Hh$ is assumed to be blurred by a linear operator $L: \Hh \to \Hh$ (e.g. the camera is out of focus or in movement during the exposure). Furthermore, it is corrupted with a noise $\nu$, so that only the result $b = Lx + \nu$ is known to us. We want to reconstruct the original image $x$ by considering the minimisation problem
\[
  \min_{x \in \Hh} \p{\frac{\mu}{2} \norm{Lx - b}^2 + J\p{Dx}},
\]
where $\norm{\cdot}$ denotes the usual Euclidean norm, $\mu > 0$ is a regularisation parameter, $D: \RR^{mn} \to \RR^{2mn}$ is the discrete gradient operator given by $Dx = \p{K_1x, K_2x}$, where
\begin{align*}
  K_1: \Hh &\to \Hh, \p{K_1 x}_{i, j} \defeq \begin{cases}x_{i + 1, j} - x_{i, j},& i = 1, \ldots, m-1; j = 1, \ldots, n; \\ 0,& i = m; j = 1, \ldots, n\end{cases} \\
  K_2: \Hh &\to \Hh, \p{K_2 x}_{i, j} \defeq \begin{cases}x_{i, j + 1} - x_{i, j},& i = 1, \ldots, m; j = 1, \ldots, n - 1; \\ 0,& i = 1,\ldots, m; j = n,\end{cases}
\end{align*}
and $J: \Hh \to\RR$ is a regularising functional penalising noisy images. We want to compare several choises of the functional $J$ proposed by \cite{GassoRakotomamonjyCanu:2009, LouZengOsherXin:2015}, all of which have in common that they want to induce sparsity of $Dx$, i.e. having many components equal to zero.

The \emph{smoothly clipped absolute deviation} (SCAD) penalty was introduced by Fan and Li in \cite{FanLi:2001}. It is defined by
\[
  \mathrm{SCAD}_{\lambda, a}\p{z} = \sum_{j = 1}^{2mn} g_{\lambda, a}\p{z_j},
\]
where $\lambda > 0$, $a > 1$ and
\begin{align*}
  g_{\lambda, a}\p{z_j} &= \begin{cases}\lambda \abs{z_j} &\text{if } \abs{z_j} \leq \lambda, \\ \frac{-\abs{z_j}^2 + 2a\lambda \abs{z_j} - \lambda^2}{2\p{a-1}} &\text{if } \lambda < \abs{z_j} \leq a\lambda, \\ \frac{\p{a + 1} \lambda^2}{2} &\text{if } \abs{z_j} > a\lambda.\end{cases} \\
  &= \lambda \abs{z_j} - \begin{cases}0 &\text{if } \abs{z_j} \leq \lambda, \\ \frac{\p{\abs{z_j} - \lambda}^2}{2\p{a-1}} &\text{if } \lambda < \abs{z_j} \leq a\lambda, \\ \lambda\abs{z_j} - \frac{\p{a + 1} \lambda^2}{2} &\text{if } \abs{z_j} > a\lambda.\end{cases}
\end{align*}
Denoting the part after the curly brace as $h_{\lambda, a}\p{z_j}$ and $h_{\lambda, a}\p{z}:= \sum_{j=1}^{2mn} h_{\lambda, a}\p{z_j}$, we have
\[
  \Prox_{\gamma h_{\lambda, a}^*}\p{z} = \begin{cases}-\lambda & \text{if } z \leq -\p{1 + \gamma a}\lambda, \\ -\frac{z + \gamma\lambda}{1 + \gamma a - \gamma}& \text{if } -\p{1 + \gamma a}\lambda \leq z \leq -\gamma\lambda, \\ 0 & \text{if } \abs{z} \leq \gamma\lambda, \\ \frac{z - \gamma\lambda}{1 + \gamma a - \gamma}& \text{if } \gamma\lambda \leq z \leq \p{1 + \gamma a}\lambda, \\ \lambda &\text{if } z \geq \p{1 + \gamma a}\lambda.
  \end{cases}
\]

The \emph{Zhang penalty} \cite{Zhang:2009} is defined by
\[
  \mathrm{Zhang}_{a}\p{z} = \sum_{j = 1}^{2mn} g_{a}\p{z_j},
\]
where $a > 0$ and
\begin{align*}
  g_{a} \p{z_j} &= \begin{cases}\frac{1}{a}\abs{z_j}& \text{if } \abs{z_j} < a, \\ 1& \text{if } \abs{z_j} \geq a.\end{cases} \\
  &= \frac{1}{a} \abs{z_j} - \begin{cases}0& \text{if } \abs{z_j} < a, \\ \frac{1}{a} \p{\abs{z_j} - a}& \text{if } \abs{z_j} \geq a.\end{cases}
\end{align*}
Denoting the part after the curly brace as $h_{a}\p{z_j}$ and $h_{a}\p{z}:= \sum_{j=1}^{2mn} h_{a}\p{z_j}$, we have
\[
  \Prox_{\gamma h_a^*}\p{z} = \begin{cases}-\frac{1}{a}& \text{if } z \leq -\frac{1}{a} - \gamma a, \\ z + \gamma a & \text{if } -\frac{1}{a} - \gamma a \leq z \leq -\gamma a, \\ 0& \text{if } -\gamma a \leq z \leq \gamma a, \\ z - \gamma a &\text{if } \gamma a \leq z \leq \frac{1}{a} + \gamma a, \\ \frac{1}{a} &\text{if }z \geq \frac{1}{a} + \gamma a.\end{cases}
\]

The \emph{LZOX penalty} \cite{LouZengOsherXin:2015} is defined by
\[
  \mathrm{LZOX}_a\p{z} = \norm{Dx}_{\ell^1} - a \norm{Dx}_{\times},
\]
where $\norm{\cdot}_{\ell^1}$ denotes (as usual) the sum of the absolute values and
\[
  \norm{\p{u, v}}_\times \defeq \sum_{i = 1}^m \sum_{j = 1}^n \sqrt{u_{i, j}^2 + v_{i, j}^2},
\]
where $y = \p{u, v}$ is the splitting according to the definition of $D$. The algorithm \eqref{eq:FBDC_alg:1}--\eqref{eq:FBDC_alg:2} can now be applied to any of the models described above, since the models are written as d. c. problems and the components are easily accessible for computation, with the exception of the function $\norm{\cdot}_{\ell^1} \circ D$, see \cite{BotHendrich:2014}. For the latter, see the following section.

\subsection{The proximal point of the anisotropic total variation}\label{sec:submethod}
In order to apply Algorithm \eqref{eq:FBDC_alg:1}--\eqref{eq:FBDC_alg:2} to any of the problems, we have to calculate the proximal point of the anisotropic total variation by solving the optimization problem
\begin{equation}\label{eq:aniso_prox}
  \inf \setcond{\frac{1}{2\gamma} \norm{x - b}^2 + \norm{Dx}_{\ell^1}}{x \in \Hh}
\end{equation}
for some $\gamma > 0$ and $b \in \Hh$ in each step. The Fenchel dual problem \cite[Chapter 19]{BauschkeCombettes:2011} is given by
\begin{equation}\label{eq:aniso_prox_dual}
  \inf \setcond{\frac{\gamma}{2}\norm{D^* x^*}^2 - \inpr{b}{D^* x^*}}{x^* \in \Gg, \norm{x^*}_{\ell^\infty} \leq 1}.
\end{equation}
Instead of solving \eqref{eq:aniso_prox}, we could also solve \eqref{eq:aniso_prox_dual} (see \cite{BeckTeboulle:2014}), as the following result shows.

\begin{lemma}
  Let $x^* \in \Gg$ be a solution of \eqref{eq:aniso_prox_dual}. Then $x = b - \gamma D^* x^*$ is a solution of \eqref{eq:aniso_prox}.
\end{lemma}
\begin{proof}
  See \cite[Example 19.7]{BauschkeCombettes:2011}. In short:
  \begin{align*}
    0 \in D\p{\gamma D^* x^* - b} + \partial \norm{\cdot}_{\ell^1}^*\p{x^*} &\implies D^* x^* \in D^* \partial\norm{\cdot}_{\ell^1}\p{D\p{b - \gamma D^* x^*}} \\
    &\implies \frac{1}{\gamma} \p{b - x} \in D^* \partial \norm{\cdot}_{\ell^1} \p{Dx}\\
    &\Longleftrightarrow 0 \in \partial \left(\frac{1}{2\gamma} \norm{(\cdot) - b}^2 + \norm{D (\cdot)}_{\ell^1} \right)(x). \qedhere
  \end{align*}
\end{proof}
To the formulation \eqref{eq:aniso_prox_dual}, the Forward-Backward algorithm can be applied, since the objective function  is differentiable and the feasible set is easy to project on.

%The \emph{duality gap} is then the sum of the primal and dual objective value at the primal and dual estimations $x^* \in \Gg$ with $\norm{x^*}_{\ell^\infty} \leq 1$ and $x \in \Hh$, respectively.
%\begin{align*}
%  \mathcal D = \frac{1}{2\gamma} \norm{x - b}^2 + \norm{Dx}_{\ell^1} + \frac{\gamma}{2} \norm{D^* x^*}^2 - \inpr{b}{D^* x^*}.
%\end{align*}

\subsection{Numerical results}
We implemented the FBDC algorithm applied to the model described above and tested the MATLAB code on a PC with Intel Core i5 4670S ($4\times$ 3.10GHz) and 8GB DDR3 RAM (1600MHz). Our implementation used the method described in Section \ref{sec:submethod} until the $\ell^\infty$ distance between two iterations was smaller than $10^{-4}$. Both stepsizes were chosen as $\mu_n = \gamma_n = \frac{1}{8\mu}$ for all $n\geq 0$. As initial value, we chose $x_0 = b$ and picked $v_0 \in \partial h\p{Kx_0}$.

We picked the image \texttt{texmos3} from \url{http://sipi.usc.edu/database/database.php?volume=textures&image=64} and convolved it with a Gaussian kernel with 9 pixels standard devitation. Afterwards we added white noise with standard deviation $50/255$, projected the pixels back to the range $\clint{0}{1}$ and saved the image in TIFF format, rounding the brightness values to multiples of $1/255$. See Figure \ref{fig:images_LZOX} for original, blurry and reconstructed image.

The \emph{improvement in signal-to-noise ratio} or \emph{ISNR value} of a reconstruction is given by
\[
  \mathrm{ISNR}\p{x_k} = 10 \log_{10} \p{\frac{\norm{x - b}^2}{\norm{x - x_k}^2}},
\]
where $x$ is the (usually unknown) original, $b$ is the known blurry and noisy and $x_k$ is the reconstructed image. For the ISNR values after 50 iterations, see Tables \ref{tab:LZOX} and \ref{tab:Zhang}. The development of the ISNR values over the iterations is shown in Figure \ref{fig:development}.

We see that the nonconvex models provide reasonable reconstructions of the original image and the best numerical performance for this particular choice of the stepsizes and the number of iterations is not achieved for the convex model (LZOX with $\alpha = 0$), but for the nonconvex models.
\begin{table}\centering
  \begin{tabular}{cccccccc} \toprule
    & $\alpha = 0.00$ & $\alpha = 0.2$ & $\alpha = 0.4$ & $\alpha = 0.5$ & $\alpha = 0.6$ & $\alpha = 0.8$ & $\alpha = 1.0$ \\ \midrule
    $\mu = \phantom{000}1.0$  & $-3.0288$ & $-4.2266$ & $-3.7637$ & $-3.6569$ & $-3.5150$ & $-4.3590$ & $-13.701$ \\
    $\mu = \phantom{00}10.0$ & $5.9227$ & $6.26615$ & $6.414791$ & $6.44871$ & $6.45780$ & $6.28863$ & $4.301090$ \\
    $\mu = \phantom{00}20.0$ & $6.76613$ & $6.90005$ & $\mathbf{6.93064}$ & $6.917926$ & $6.88018$ & $6.61521$ & $5.305623$ \\
    $\mu = \phantom{00}50.0$ & $6.81752$ & $6.78308$ & $6.65411$ & $6.4923$ & $6.36250$ & $5.780558$ & $4.741993$ \\
    $\mu = \phantom{0}100.0$ & $5.29597$ & $5.23264$ & $5.05189$ & $4.91247$ & $4.739717$ & $4.287092$ & $3.696120$ \\
    $\mu = \phantom{0}200.0$ & $3.088196$ & $3.060511$ & $2.985871$ & $2.930448$ & $2.863122$ & $2.693096$ & $2.477708$ \\
    $\mu = \phantom{0}500.0$ & $1.317390$ & $1.312168$ & $1.298834$ & $1.288983$ & $1.277010$ & $1.246724$ & $1.208036$ \\
    $\mu = 1000.0$ & $0.692487$ & $0.691049$ & $0.687585$ & $0.685057$ & $0.682000$ & $0.674272$ & $0.664401$ \\ \bottomrule
  \end{tabular}
  \caption{LZOX after 50 iterations}\label{tab:LZOX}
\end{table}

\begin{table}\centering
  \begin{tabular}{ccccccc} \toprule
    & $\alpha = 0.01$ & $\alpha = 0.03$ & $\alpha = 0.1$ & $\alpha = 0.3$ & $\alpha = 1.0$ & $\alpha = 3.0$ \\ \midrule
    $\mu = \phantom{000}1.0$ & $-43.708$ & $-33.711$ & $-23.148$ & $-13.846$ & $-3.0288$ & $2.4922$ \\
    $\mu = \phantom{00}10.0$ & $-18.781$ & $-9.9406$ & $-3.2070$ & $2.5442$ & $5.9227$ & $\mathbf{6.97777}$ \\
    $\mu = \phantom{00}20.0$ & $-11.270$ & $-4.8428$ & $0.43533$ & $4.7768$ & $6.76613$ & $6.57299$ \\
    $\mu = \phantom{00}50.0$ & $-4.8333$ & $-1.05553$ & $2.63959$ & $6.46109$ & $6.81752$ & $3.952101$ \\
    $\mu = \phantom{0}100.0$ & $-1.7546$ & $-0.14560$ & $3.16532$ & $6.90202$ & $5.29597$ & $2.129705$ \\
    $\mu = \phantom{0}200.0$ & $-0.41418$ & $0.0619477$ & $2.98543$ & $6.38513$ & $3.088196$ & $1.110186$ \\
    $\mu = \phantom{0}500.0$ & $0.0077144$ & $0.121807$ & $2.101321$ & $3.816813$ & $1.317390$ & $0.482406$ \\
    $\mu = 1000.0$ & $0.0528014$ & $0.127592$ & $1.423684$ & $2.070959$ & $0.692487$ & $0.271777$ \\ \bottomrule
  \end{tabular}
  \caption{Zhang after 50 iterations}\label{tab:Zhang}
\end{table}

\begin{figure}\centering
  \subfigure[]{
    \begin{tikzpicture}
      \begin{axis}[title=LZOX model,
          xlabel={iterations},
          ylabel={ISNR},
          legend entries={{$\mu = 20$, $\alpha = 0.4$}, {$\mu = 50$, $\alpha = 0\phantom{.0}$}, {$\mu = 20$, $\alpha = 0\phantom{.0}$}, {$\mu = 20$, $\alpha = 1\phantom{.0}$}},
        legend style ={at={(0.5, -0.17)}, anchor = north}]
        \addplot[blue] table {LZOX1.table};
        \addplot[green] table {LZOX2.table};
        \addplot[red] table {LZOX3.table};
        \addplot[black] table {LZOX4.table};
      \end{axis}
    \end{tikzpicture}
  }
  \subfigure[]{
    \begin{tikzpicture}
      \begin{axis}[title=Zhang model,
          xlabel={iterations},
          ylabel={ISNR},
          legend entries={{$\mu = 10$, $\alpha = 3\phantom{.0}$}, {$\mu = 20$, $\alpha = 1\phantom{.0}$}, {$\mu = 50$, $\alpha = 0.3$}, {$\mu = 20$, $\alpha = 1\phantom{.0}$}},
        legend style ={at={(0.5, -0.17)}, anchor = north}]
        \addplot[blue] table {Zhang1.table};
        \addplot[green] table {Zhang2.table};
        \addplot[red] table {Zhang3.table};
%        \addplot[black] table {Zhang4.table};
      \end{axis}
    \end{tikzpicture}
  }
  \caption{ISNR values vs. iterations}\label{fig:development}
\end{figure}
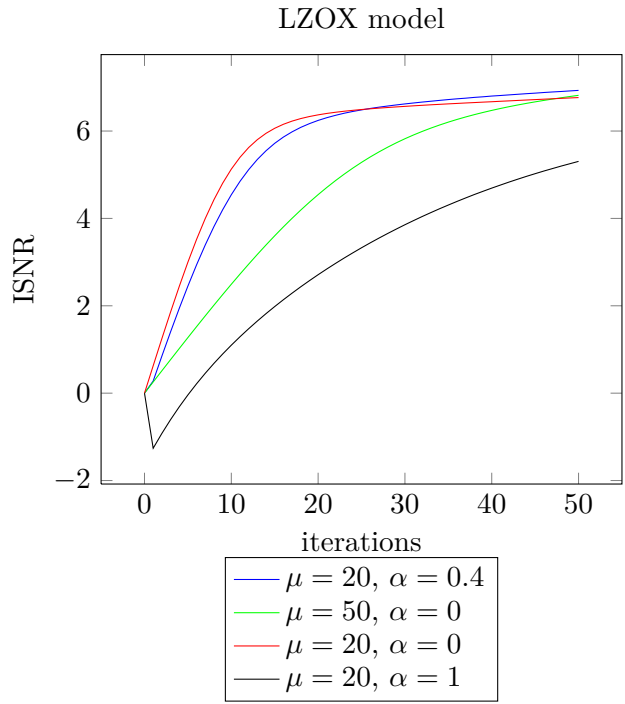
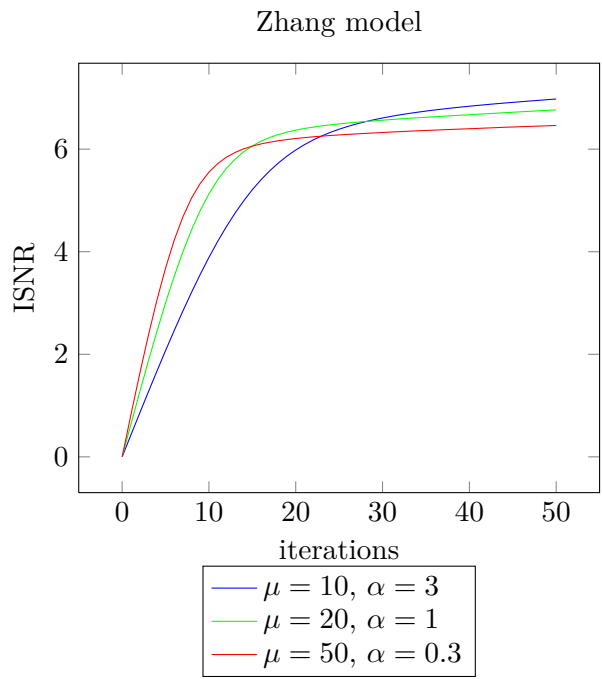
\begin{figure} \centering
  \subfigure[Original image]{\includegraphics[width=0.3\textwidth]{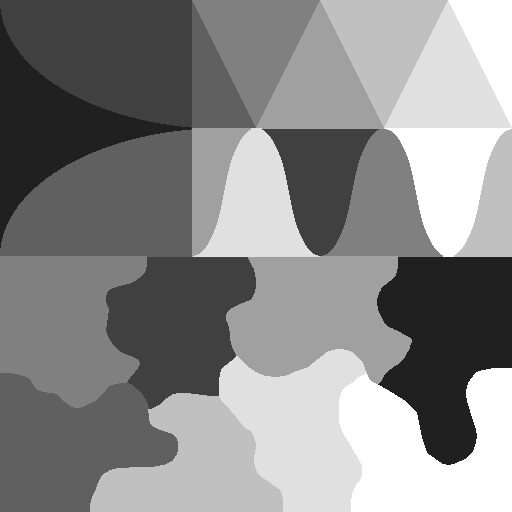}}
  \subfigure[Blurry image]{\includegraphics[width=0.3\textwidth]{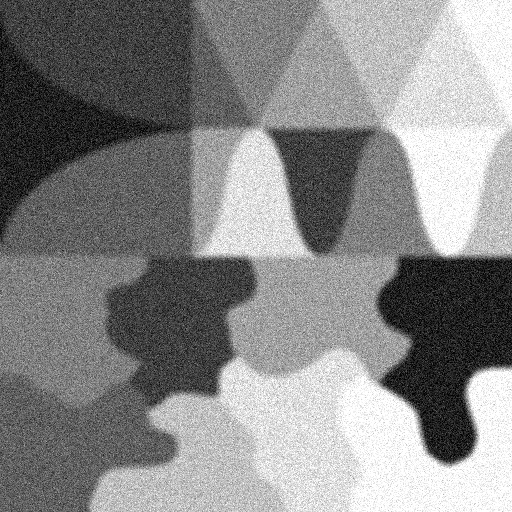}}
  \subfigure[LZOX, $\mu = 20$, $\alpha = 0.4$]{\includegraphics[width=0.3\textwidth]{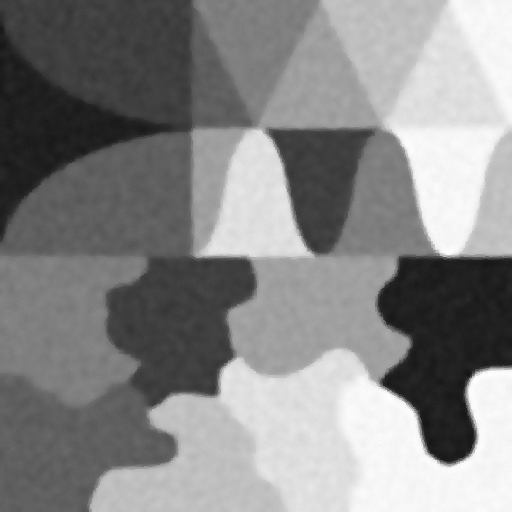}}
  \subfigure[LZOX, $\mu = 20$, $\alpha = 1$]{\includegraphics[width=0.3\textwidth]{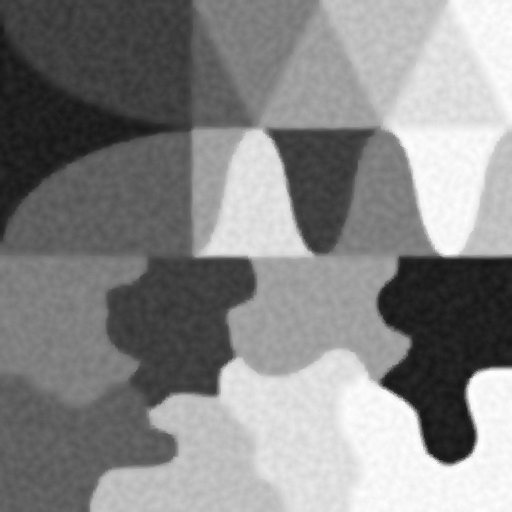}}
  \subfigure[LZOX, $\mu = 50$, $\alpha = 0$]{\includegraphics[width=0.3\textwidth]{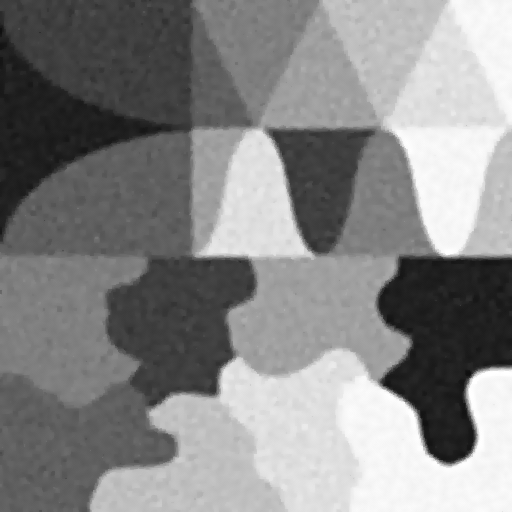}}
  \subfigure[Zhang, $\mu = 10$, $\alpha = 3$]{\includegraphics[width=0.3\textwidth]{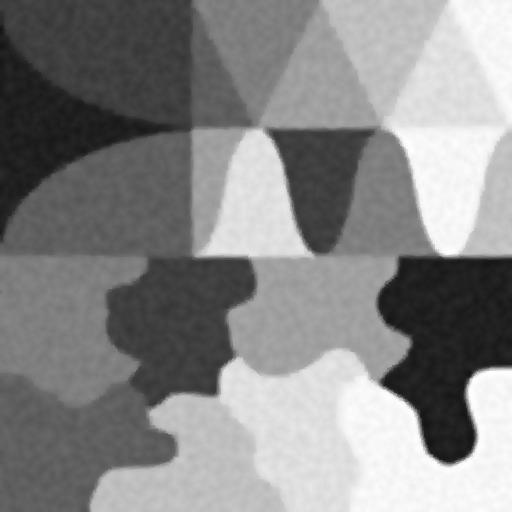}}
  \subfigure[Zhang, $\mu = 20$, $\alpha = 1$]{\includegraphics[width=0.3\textwidth]{texmos3_20_1.png}}
  \subfigure[Zhang, $\mu = 100$, $\alpha = 0.1$]{\includegraphics[width=0.3\textwidth]{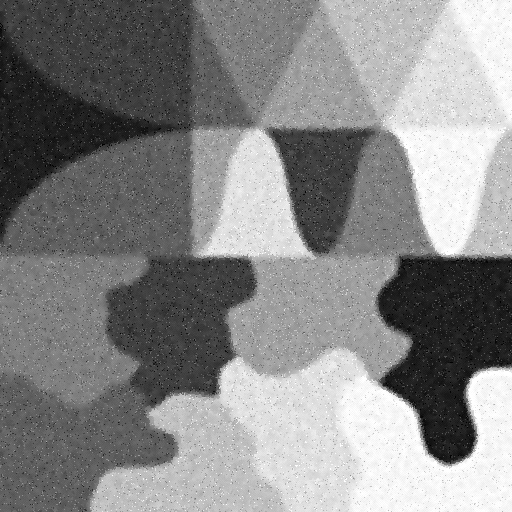}}
  \caption{Original image, blurry and noisy image and reconstructions.}\label{fig:images_LZOX}
\end{figure}

\section{Acknowledgements}
The authors are grateful to Joseph Salmon for pointing their attraction on the paper \cite{GassoRakotomamonjyCanu:2009}.

\end{document}